\documentclass[11pt]{amsart}

\usepackage{a4wide}
\usepackage{amssymb, setspace}
\usepackage[all]{xy}
\SelectTips{cm}{}

\usepackage{amsrefs}

\usepackage{hyperref}

\newtheorem{theorem}{Theorem}[section]
\newtheorem{lemma}[theorem]{Lemma}

\newtheorem{corollary}[theorem]{Corollary}

\theoremstyle{definition}
\newtheorem{definition}[theorem]{Definition}
\theoremstyle{remark}
\newtheorem{remark}[theorem]{Remark}
\theoremstyle{remark}
\newtheorem{example}[theorem]{Example}
\theoremstyle{remark}

\theoremstyle{remark}

\newcommand{\cC}{{\mathcal C}}

\newcommand{\Zb}{{\mathbb Z}}

\newcommand{\Rb}{{\mathbb R}}
\newcommand{\Nb}{{\mathbb N}}
\newcommand{\map}{{\rm map}}
\newcommand{\aut}{\operatorname{Aut}}
\newcommand{\gl}{\operatorname{GL}}
\newcommand{\id}{\operatorname{id}}

\newcommand{\homeo}{\operatorname{Homeo}}

\renewcommand{\phi}{\varphi}

\begin{document}

\title{Cantor systems and quasi-isometry of groups}

\author{Kostya Medynets}
\address{K.M., Department of Mathematics, U.S. Naval Academy,  Annapolis, MD 21402, United States.}
\email{medynets@usna.edu}

\author{Roman Sauer}
\address{R.S., Institute of Algebra and Geometry, Karlsruhe Institute of Technology, 76131 Karlsruhe, Germany.}
\email{roman.sauer@kit.edu}

\author{Andreas Thom}
\address{A.T., Institute of Geometry, TU Dresden, 01069 Dresden, Germany.}
\email{andreas.thom@tu-dresden.de}


\onehalfspace

\begin{abstract}
The purpose of this note is twofold. In the first part we observe that two finitely generated non-amenable groups are quasi-isometric if and only if they admit topologically orbit equivalent Cantor minimal actions. In particular, free groups of different rank admit topologically orbit equivalent Cantor minimal actions -- unlike in the measurable setting. In the second part we introduce the measured orbit equivalence category of a Cantor minimal system and construct (in certain cases) a representation of this category on the category of finite-dimensional vector spaces. This gives rise to novel fundamental invariants of the orbit equivalence relation together with an ergodic invariant probability measure.
\end{abstract}

\maketitle

\tableofcontents

\section{Introduction}

This note is about a relationship between ergodic theory of group actions and geometric group theory. In geometric group theory, finitely generated groups are typically studied from a coarse point of view as metric spaces up to quasi-isometry -- a weakening of the usual notion of isomorphism or commensurability. On the other side, in the study of group actions on compact metric spaces, the notion of topological orbit equivalence arises in a natural way as a weakening of the notion of conjugacy of actions. Based on Gromov's dynamical characterization of quasi-isometry~\cite{gromov}*{0.2.C2′} we show that two finitely generated non-amenable groups are quasi-isometric if and only if they admit topologically orbit equivalent Cantor minimal systems. As a particular consequence, we can conclude that non-abelian free groups of different rank admit topologically orbit equivalent Cantor minimal systems. This is in harsh contrast to the analogous measurable setting, where the measurable orbit equivalence relation (of probability measure preserving systems) remembers the rank of the free group, see Gaboriau's seminal work~\cite{MR1919400}.

In the second part of this paper, starting with Section \ref{sec: orcat}, we study self-orbit equivalences of a Cantor minimal system. To this end, we introduce a category $m{\mathcal C}(R)$ which is called the measured orbit equivalence category. Its objects are group actions that induce an isomorphic orbit equivalence relation $R$ endowed with an ergodic invariant probability measure. Morphisms are suitable intertwiners given by a measure-preserving homeomorphism and a cocycle, see Section~\ref{sec: orcat}. In Section~\ref{sec: functors}  we then construct various functors from the category $m{\mathcal C}(R)$ to the category of vector spaces. In the case that $R$ is induced by a free action of a nilpotent group, these functors land in the category of finite-dimensional vector spaces and thus yield a linear representation of the category $m{\mathcal C}(R)$. In the last two sections we study in more detail the special situation when the nilpotent group is just $\Zb^d$ and also consider the realization problem, i.e., determine the potential image of the functor. The main results from these sections are Theorems~\ref{thm: full group representation in nilpotent case}, \ref{thm: full group representation in abelian case} and \ref{thm: realization}.
In Section~\ref{sec: f	d} we end the discussion with an outline of further directions.

\section{Basic definitions}

Let $X$ stand for a standard Cantor set, e.g.~$X=\{0,1\}^{\mathbb N}$. 
Let $\Gamma$ be a (not necessarily countable) group, and let $\pi$ be a faithful action of $\Gamma$ on $X$ by homeomorphisms, that is, a monomorphism $\pi\colon\Gamma\hookrightarrow\homeo(X)$. 
We denote this data by $(X,\Gamma, \pi)$ and call it a \emph{Cantor dynamical system}. Since we are only considering faithful actions, we may and will identify $\Gamma$ with its image in $\homeo(X)$ and denote the dynamical system 
-- by slight abuse of notation -- just by $(X,\Gamma)$. Accordingly, 
we often denote the group action by $(g,x) \mapsto gx$, suppressing $\pi$ in the 
notation when the group action is clear from the context. 

We denote by
$[[X, \Gamma]]$ or $[[\Gamma]]$ (depending on whether $X$ is clear from the context)
the set of all homeomorphisms $T$ of $X$ for which there exists $k \in \mathbb N$, a clopen partition  of $X= A_1\sqcup\ldots \sqcup A_k$, and group elements
$g_1,\ldots,g_k\in \Gamma$ such that $T|_{A_i} = g_i|_{A_i}$ for every $i=1,\ldots,k$. It is clear that $[[\Gamma]]$ is again a group -- it is called the \emph{topological full group} of the dynamical system $(X,\Gamma)$.
A clopen set $U\subset X$ is called \emph{wandering} if $gU\cap U= \emptyset$ holds for all group elements $g\neq 1$.

In many natural situations, the topological full group is a complete invariant of the topological orbit equivalence relation, see Definition \ref{toporbeq} for terminology. This can be made precise as follows:

\begin{theorem}[see \cite{Me2011}]  \label{theomed}
Let $(X,\Gamma)$ and $(Y,\Lambda)$ be Cantor dynamical systems. Assume that the dynamical system $(X,\Gamma)$ and $(Y,\Lambda)$ have no wandering clopen sets and no orbits of length one and two.
Suppose that there is a group isomorphism $\eta\colon [[X,\Gamma]] \rightarrow [[Y,\Lambda]]$. Then there is a homeomorphism $\phi_\eta : X\rightarrow Y$ such that $\eta(g) = \phi_\eta \circ g \circ \phi_\eta^{-1}$ for every $g\in [[X,\Gamma]]$. In particular, the dynamical system $(X,\Gamma)$ and $(Y,\Lambda)$ are topologically orbit equivalent.
\end{theorem}

\begin{remark} This result remains valid if we replace the topological full group $[[X,\Gamma]]$  by the \emph{full group} $[X,\Gamma]$, which consists of all homeomorphisms of $X$ preserving $\Gamma$-orbits.
\end{remark}

\begin{remark}\label{rem: uniqueness}
The homeomorphism $\phi_\eta$ in the previous theorem is unique, which also justifies its notation. Indeed, suppose that the homeomorphisms $\phi$ and $\psi$ satisfy $\eta(g)=\phi\circ g\circ \phi^{-1}$
and $\eta(g)=\psi\circ g\circ \psi^{-1}$ for $g\in [[X,\Gamma]]$, and suppose that $\phi\ne \psi$. Then there is an open subset $U\subset X$ such that
$\phi(U)\cap \psi(U)=\emptyset$. Possibly by decreasing $U$, we find $h=\eta(g)\in [[Y,\Gamma]]$ such that $h(x)=x$ for all $x\in \psi(U)$ and such that there exists
$\phi(x_0)\in \phi(U)$ such that $h(\phi(x_0))\ne \phi(x_0)$. In particular, we obtain that
$\psi(g(x_0))=h(\psi(x_0))=\psi(x_0)$ and $\phi(g(x_0))=h(\phi(x_0))\ne \phi(x_0)$. Since $\phi$ and $\psi$ are bijective, this is a contradiction.
\end{remark}

%
%
%
%
%

\begin{remark}
Every element $g \in [[X,\Gamma]]$ acts on $[[X,\Gamma]]$ by conjugation, i.e., there exists a homomorphism ${\rm ad} \colon [[X,\Gamma]] \to \aut([[X,\Gamma]])$ with ${\rm ad}(g)(h)=ghg^{-1}$. For $g\in [[X,\Gamma]]$ the homeomorphism $\phi_{\rm ad(g)}$ in the previous theorem is just $g$ itself considered as a homeomorphism of $X$. 
\end{remark}

Our goal is to relate geometric group theory (quasi-isometries) to topological dynamics (topological orbit equivalence). Next we recall the relevant notions. 

Let $R$ and $R'$ be equivalence relations on topological spaces $X$ and $X'$, 
respectively. Then $R$ and $R'$ are called \emph{orbit equivalent} if there is a 
homeomorphism $h\colon X\to X'$ -- called the \emph{orbit equivalence} -- such that $(h\times h)(R)=R'$. Assume additionally 
that $R$ and $R'$ are the orbit equivalence relations of free continuous actions 
of $\Gamma$ on $X$ and $\Lambda$ on $X'$. The \emph{cocycles} $\alpha\colon \Gamma\times X\to \Lambda$ and $\beta\colon \Lambda\times X'\to \Gamma$ of an orbit equivalence $h\colon X\to X'$ between $R$ and $R'$ are defined by the relations 
\begin{equation}\label{eq: cocycle relations}
\alpha(\gamma, x)h(x)=h(\gamma x)~\text{ and }~\beta(\lambda, x')h^{-1}(x')=h^{-1}(\lambda x').
\end{equation}
\begin{definition} \label{toporbeq}
	Two free continuous actions of groups $\Gamma$ on $X$ and $\Lambda$ on $X'$ are called \emph{topologically orbit equivalent} if there is an orbit equivalence $h\colon X\to X'$ of the corresponding orbit equivalence relations whose cocycles are continuous. In that case $h$ is called a \emph{topological orbit equivalence}. 
\end{definition}

Next we briefly recall some notions from geometric group theory to fix the notation.
Let $\Gamma$ be a finitely generated group and $S \subset \Gamma$ be a finite symmetric generating set. Recall that a subset $S \subset \Gamma$ is called symmetric if $S^{-1} := \{s^{-1} \mid s \in S \}$ coincides with $S$. By the choice of $S$ one
defines the \emph{length function} $\ell_S$ on $\Gamma$ by
\[\ell_S(g) := \min\{ n \in \Nb \mid \exists_{s_1,\dots,s_n \in S}~ g=s_1 \cdots s_n \}.\]
It is clear that $\ell_S(gh) \leq \ell_S(g) + \ell_S(h)$ for all $g,h \in \Gamma$, and $\ell_S(g)=0$ if and only if $g=e$. Using the length function,
the \emph{word metric} on $\Gamma$ is given by
\[d_{\Gamma}(g,h) := \ell_S(g^{-1}h).\]
Note that $d_{\Gamma}$ is left-invariant, i.e.~$d_{\Gamma}(kg,kh) = d_{\Gamma}(g,h)$ for all $g,h,k \in \Gamma$. Note also that $d_{\Gamma}$ depends on $S$, but we will suppress this dependence in the notation.

\begin{definition}
If $\Gamma$, $\Lambda$ are finitely generated groups with word metrics $d_{\Gamma}$ and $d_{\Lambda}$, then a
map $\varphi \colon \Gamma \to \Lambda$ is called
\begin{enumerate}
\item[i)] a \emph{quasi-isometry} if there exists a constant $C>0$, such that
$$C^{-1} \cdot d_{\Gamma}(g,h) -C \leq  d_{\Lambda}(\varphi(g),\varphi(h)) \leq C \cdot d_{\Gamma}(g,h) + C$$
for all $g,h \in \Gamma$, and for all $k \in \Lambda$, there exists $g \in \Gamma$ with $d_{\Lambda}(\varphi(g),k) \leq C$;
\item[ii)] a \emph{bi-Lipschitz equivalence} if $\varphi$ is bijective and there exists a constant $C$, such that
$$C^{-1} \cdot d_{\Gamma}(g,h) \leq  d_{\Lambda}(\varphi(g),\varphi(h)) \leq C \cdot d_{\Gamma}(g,h)$$
for all $g,h \in \Gamma$. We call $C$ a Lipschitz constant for $\varphi$.
\end{enumerate}
Two maps $\varphi,\psi \colon \Gamma \to \Lambda$ are said to have \emph{bounded distance} if there exists a constant $C>0$ such that $d_{\Lambda}(\varphi(g),\psi(g))\leq C$ for all $g \in \Gamma$.
\end{definition}

It is clear from the definition that two groups which are bi-Lipschitz equivalent are also quasi-isometric. The converse is not true in general, see \cite{MR2730576}. However, if one group is (and hence both are) non-amenable, the two notions agree.

\begin{theorem}[Whyte, \cite{Whyte}] \label{whyte} Let $\Gamma, \Lambda$ be non-amenable finitely generated groups. Every quasi-isometry is within bounded distance to a bi-Lipschitz equivalence.
\end{theorem}

This was shown before for trees (and hence for free groups) by Papasoglu~\cite{MR1326733}, answering a question of Gromov~\cite{gromov}.

\section{Gromov's construction}

Let $\Gamma,\Lambda$ be finitely generated groups. We consider the space $\map(\Gamma,\Lambda)$ of all maps from $\Gamma$ to $\Lambda$ with the topology of pointwise convergence. The space $\map(\Gamma,\Lambda)$ is endowed with a natural $\Gamma \times \Lambda$-action, given by the formula
\[((g,\lambda)\varphi)(h) = \lambda \varphi(g^{-1}h),  \quad \forall g,h \in \Gamma,\lambda \in \Lambda, \varphi \in \map(\Gamma,\Lambda).\]

Let $\varphi \colon \Gamma \to \Lambda$ be a bi-Lipschitz equivalence and let $C$ be a Lipschitz constant for $\varphi$. We consider the set
$$\Omega := \overline{(\Gamma \times \Lambda) \varphi} \subset \map(\Gamma,\Lambda), \quad \mbox{and} \quad X := \{ \psi \in \Omega \mid \psi(e)=e\}.$$
Note that the $\Gamma \times \Lambda$-action restricts to the closed subset $\Omega$. The following theorem, going to back to Gromov, clarifies the relation between properties of elements in $\Omega$ and the action of $\Gamma \times \Lambda$ on $\Omega$.

\begin{theorem}[cf.~\citelist{\cite{gromov}*{0.2.C2′}\cite{shalom}*{Theorem~2.1.2}}]\label{thm: gromov construction}
Let $\varphi \colon \Gamma \to \Lambda$ be as above with Lipschitz constant $C$.
\begin{enumerate}
\item The set $\Omega$ consists of bi-Lipschitz equivalences with Lipschitz constant $C$.
\item  The set $X \subset \Omega$ is open and compact and a fundamental domain for both the $\Gamma$-action and the $\Lambda$-action.
\end{enumerate}
\end{theorem}

Similar to the measurable setting, the $\Gamma\times\Lambda$-space~$\Omega$
yields a topological orbit equivalence between the induced actions of $\Gamma$ on
$X\cong \Lambda\backslash\Omega$ and $\Lambda$ on $X\cong \Gamma\backslash\Omega$.
The $\Gamma$-action on $X$ is given by the formula
\[(g \cdot \psi)(h) = \psi(g^{-1})^{-1}\psi(g^{-1}h)\]
and the $\Lambda$-action is given by
\[(\lambda \cdot \psi)(h) = \lambda \psi \left(\psi^{-1}(\lambda^{-1})h \right).\]
Note that $g \cdot \psi \neq (g,e)\psi$ in general and similarly for the $\Lambda$-action. Moreover, for each $x \in X$, we have $\Gamma \cdot x = \Lambda \cdot x$. The maps
\[\alpha \colon \Gamma \times X \to \Lambda, \quad \alpha(g,\psi) = \psi(g^{-1})^{-1}\]
and
\[\beta \colon \Lambda \times X \to \Gamma, \quad \beta(\lambda ,\psi)= \left(\psi^{-1}(\lambda^{-1}) \right)^{-1}.\]
satisfy cocycle identities
\[\alpha(gh,\psi) = \alpha(g,h \cdot \psi) \alpha(h,\psi), \quad \forall g,h \in \Gamma, \psi \in X\]
and the relation~\eqref{eq: cocycle relations} for $h=\id$. Note, however, that whilst the action of $\Gamma$ (or $\Lambda$) on $\Omega$ is free, the induced action on $X$ is not necessarily free. 

It is obvious that $\alpha$ and $\beta$ are continuous. In particular, for fixed $g \in \Gamma$, the map $\alpha(g,\_) \colon X \to \Lambda$ takes only finitely many values and defines a clopen partition of $X$, so that the action of $g$ on each piece is given by the action of some element in $\Lambda$.

\begin{theorem}\label{main1}
  Let $\Gamma$ and $\Lambda$ be finitely generated groups. Then $\Gamma$ and $\Lambda$ admit free  continuous actions on the Cantor set that are topologically orbit equivalent if and only if $\Gamma$ and $\Lambda$ are bi-Lipschitz equivalent.
\end{theorem}

The technique used in the proof below is the topological analog of the one in Furman's paper~\cite{furmanOE}*{Theorem~3.3.}. 

\begin{proof}
  (1) Suppose that $\Gamma$ and $\Lambda$ are bi-Lipschitz equivalent.  Now apply Gromov's construction above to obtain orbit equivalent actions.  In general, the action on $\Omega$ will not be free. However, due to~\cite{Hj2006} we may pick some free Cantor action $(Y,\Gamma \times \Lambda)$ and consider the diagonal action on $\Omega':=\Omega \times Y$ instead. It is clear that $X':=X \times Y$ is an open and compact fundamental domain for both the $\Gamma$-action and the $\Lambda$-action. Thus, $(\Gamma,X')$ and $(\Lambda,X')$ are free actions inducing the same orbit structure by the preceding paragraph. Forcing freeness in that way is an idea from~\cite{MR1919400}*{Theorem~2.3}. Note that we only relied on the second property of Theorem~\ref{thm: gromov construction} in the preceding paragraph and did not care about the origin of $\Omega$.

We may now pass to a minimal subsystem $Z \subset X'$ for both the $\Gamma$-action and the $\Lambda$-action -- as minimality is a property of the orbit structure. It is now apparent that $(\Gamma,Z)$ and $(\Lambda,Z)$ are topologically orbit equivalent Cantor minimal systems. This finishes the proof of the backward implication. 

(2)  Suppose that we have two finitely generated groups $\Gamma$ and $ \Lambda$, a Cantor set $X$, and topologically orbit equivalent free continuous actions $\Gamma \curvearrowright X$ and $\Lambda \curvearrowright X$. The orbit equivalence cocycle
$$\alpha \colon \Gamma \times X \to \Lambda$$ is then continuous. Let us pick a point $x \in X$ and set $\varphi(g) := \alpha(g^{-1},x)^{-1}$.

Let $S$ be a finite symmetric generating set for $\Gamma$ and $S'$ be a finite symmetric generating set of $\Lambda$. Let $g,h \in \Gamma$ be arbitrary. If $d_{\Gamma}(g,h)=n$, then $g^{-1}h = s_1 \cdots s_n$ for $s_1,\dots,s_n \in S$. First of all, note that
$$e=\alpha(hh^{-1},x) = \alpha(h,h^{-1}x)\alpha(h^{-1},x)$$
so that $\alpha(h^{-1},x)^{-1} = \alpha(h,h^{-1}x)$. Note that $\alpha(S \times X) \subset \Lambda$ is finite and set:
$$C:= \max \left\{ \ell_{S'}(\lambda) \mid \lambda \in \alpha(S \times X) \right\}.$$ We are now ready to prove the first inequality:
\begin{align*}
d_{\Lambda}(\varphi(g),\varphi(h)) = d_{\Lambda}(\alpha(g^{-1},x)^{-1},\alpha(h^{-1},x)^{-1})
&= \ell_{S'} (\alpha(g^{-1},x) \alpha(h,h^{-1}x)) \\
&= \ell_{S'} (\alpha(g^{-1}h,h^{-1}x)) \\
&= \ell_{S'} (\alpha(s_1 \cdots s_n,h^{-1}x)) \\
&\leq  \sum_{j=1}^n \ell_{S'} (\alpha(s_j, s_{j+1} \cdots s_n h^{-1}x)) \\
& \leq  C \cdot n = C \cdot d_{\Gamma}(g,h).
\end{align*}
The other inequality follows (possibly with a different constant) by symmetry. This finishes the proof.
\end{proof}

\begin{corollary} 
Let $\Gamma, \Lambda$ be non-amenable finitely generated groups. If the groups $\Gamma$ and $\Lambda$ are quasi-isometric, then there exist topologically orbit equivalent free Cantor minimal systems of $\Gamma$ and $\Lambda$.
\end{corollary}
\begin{proof}
Use Theorem \ref{whyte} to construct a bi-Lipschitz equivalence between $\Gamma$ and $\Lambda$ and apply the previous result. 
\end{proof}

Since finitely generated non-abelian free groups are bi-Lipschitz equivalent to each other~\cite{MR1326733}, we immediately obtain that they admit topologically orbit equivalent actions.

\begin{corollary}
Finitely generated non-abelian free groups of different rank have topologically 	orbit equivalent free actions.
\end{corollary}
Note that the actions in the previous corollary cannot have an invariant finite Borel measure. Indeed, any preserved measure would yield a measurable orbit equivalence and thus imply that the cost and the $\ell^2$-Betti numbers 
of the two groups would agree according to Gaboriau's work in~\cites{mercurial, gaboriau-betti}. 

Topological orbit equivalence as in Theorem~\ref{main1} and thus bi-Lipschitz equivalence of finitely generated groups is a strong condition that implies quasi-isometry. This is known to imply, for example, that cohomological dimensions agree provided they are finite~\cite{sauer} and that (under suitable finiteness conditions) for all $k \in \Nb$ the $k$-th $\ell^2$-Betti number of one groups vanishes if and only if it vanishes for the other due to an unpublished result of Pansu (see also~\cite{sauer+schroedl}). In the next section, we will discuss how symmetries of the orbit equivalence relation yield symmetries of certain cohomological invariants of the acting groups.

As a consequence of Theorem \ref{main1} we can now apply quasi-isometry rigidity results to the study of topological orbit equivalence of groups. A group $\Gamma$ is called \emph{quasi-isometrically rigid} if and only if every  group that is quasi-isometric to $\Gamma$ is already virtually isomorphic to $\Gamma$. 
Further, a class $C$ of groups is \emph{quasi-isometrically rigid} if every group that is quasi-isometric to a group in $C$ is 
virtually isomorphic to a group in $C$. This applies (by the work of many authors) to free groups, free abelian groups, the class of nilpotent groups, the class of fundamental groups of closed (compact, without boundary) surfaces, the class of fundamental groups of closed (compact, without boundary) $3$-dimensional manifolds, the class of finitely-presentable groups, the class of hyperbolic groups, the class of amenable groups, the class of fundamental groups of closed $n$-dimensional hyperbolic manifolds, the class of discrete co-compact subgroups in a simple non-compact Lie group, and many more examples~\cite{MR3329727}.

In the amenable case, bi-Lipschitz equivalence is a stronger condition than quasi-isometry, so that even finer information about topological orbit equivalence of different amenable groups can be obtained. See~\cite{MR2730576} for a discussion of bi-Lipschitz equivalence versus quasi-isometry for amenable groups.

\section{The orbit equivalence category}
\label{sec: orcat}
In this section we study two novel invariants of an orbit equivalence relation that we will call the \emph{topological and measured orbit equivalence category}, respectively. In certain cases, we obtain a representation of this category by measure preserving linear transformations on certain finite-dimensional cohomological invariants of the acting groups.

\begin{definition}
	The category $m\cC$ is defined as follows. Its objects are quadruples $(\Gamma,X,\pi,\mu)$, where $\pi \colon \Gamma \times X \to X$ is a free minimal action of $\Gamma$ on a Cantor space $X$ and $\mu$ is an ergodic $\Gamma$-invariant probability measure on $X$. A morphism 
	    between $(\Gamma,X,\pi,\mu)$ and $(\Lambda,Y,\sigma,\nu)$ is given by a pair $(\varphi,\alpha)$ where $\varphi\colon X \to Y$ is a homeomorphism and $\alpha \colon \Gamma \times X \to \Lambda$ is a continuous map (cocycle) such that $\varphi(\pi(g,x)) = \sigma(\alpha(g,x),\varphi(x))$, for all $g \in \Gamma, x \in X$, and $\varphi_*(\mu)=\nu$. 
	     The composition of morphisms $(\varphi,\alpha) \colon (\Gamma,X,\pi,\mu) \to (\Lambda,Y,\sigma,\nu)$ and $(\psi,\beta) \colon (\Lambda,Y,\sigma,\nu) \to (\Sigma,Z,\tau,\kappa)$ is given by $(\psi \circ \varphi, (g,x) \mapsto \beta(\alpha(g,x),\varphi(x)))$. Sometimes we suppress $\pi$ in the notation and denote an object in $m{\mathcal C}(R)$ by $(\Gamma,X,\mu)$. 
	     
	     The category ${\mathcal C}$ is similarly defined as $m\cC$ except that the objects are not endowed with an invariant measure and the condition that morphisms should be measure preserving is omitted. 
	     There is an obvious forgetful functor $q\colon m\cC\to \cC$. 
\end{definition}

\begin{definition}
Let $R$ be the orbit equivalence relation of a free minimal action $\pi$ of a group $\Gamma$ on a Cantor space $X$. We denote by $\cC(R)$ the full subcategory 
in $\cC$ of objects isomorphic to $(\Gamma, X,\pi)$ and call it the \emph{topological orbit equivalence category} of $R$. We denote by $m\cC(R)$ the 
full subcategory in $m\cC$ of objects that are mapped to $\cC(R)$ under the forgetful functor $q$. We call $m\cC(R)$ the \emph{measured topological orbit equivalence category} of $R$. 
\end{definition}

\begin{remark}
Let us assume that $R$ is the 
orbit equivalence relation of an action by an amenable group. Let $(\phi, \alpha)$ be a morphism in $\cC(R)$ from $(\Gamma, X, \pi)$ to $(\Lambda, Y,\sigma)$. Then both groups 
$\Gamma$ and $\Lambda$ are amenable. Choose a $\Gamma$-invariant probability Borel 
measure $\mu$ on $X$, and let $\nu:=\phi_\ast\mu$ be the push-forward measure on $Y$. Then $(\phi, \alpha)$ constitutes a morphism in $m\cC(R)$. Hence $q$ is surjective on objects and morphisms, i.e., essentially surjective and full.
\end{remark}

Let $\Gamma$ be a countable group and let $\Gamma$ act minimally and freely on a Cantor set $X$. We consider the group $\aut([[X,\Gamma]])$ of automorphisms of $[[X,\Gamma]]$.
Let $\eta \colon [[X,\Gamma]] \to [[X,\Gamma]]$ be an automorphism.
By Theorem~\ref{theomed} and Remark~\ref{rem: uniqueness}
there is a unique homeomorphism $\varphi_\eta \colon X \to X$ such that $\varphi_\eta(g(x))=\eta(g)(\varphi_\eta(x))$
for all $x \in X$ and $g\in [[X,\Gamma]]$.
We say that $\varphi_\eta$ \emph{realizes} $\eta$. We also obtain a
continuous cocycle $\alpha_\eta: \Gamma\times X\to \Gamma$ defined by
\[ \phi_\eta(gx)=\alpha(g,x)\phi_\eta(x). \]
It satisfies the \emph{cocycle identity}
\[
\alpha_\eta(g_1g_2,x)=\alpha_\eta(g_1,g_2x)\alpha_\eta(g_2,x).
\]

Now suppose there is a $\Gamma$-invariant
probability Borel measure $\mu$ on $X$. If $(\varphi_\eta)_\ast\mu=\mu$, then we say that $\eta\in\aut([[X,\Gamma]])$  is \emph{$\mu$-preserving}. The subset $\aut_\mu([[X,\Gamma]])\subset\aut([[X,\Gamma]])$ of $\mu$-preserving automorphisms is a subgroup. Note that our analysis shows that $\aut_\mu([[X,\Gamma]])$ is exactly the automorphism group of the object $(\Gamma,X,\mu)$ in the category $m{\mathcal C}(R)$. Similarly, the automorphism group of the object $(\Gamma,X)$ in ${\mathcal C}(R)$ is equal to ${\rm Aut}([[X,\Gamma]])$.

\begin{remark}
If $(X,\Gamma)$ is uniquely ergodic, then $\aut_\mu([[X,\Gamma]])=\aut([[X,\Gamma]])$: Let $\eta\in\aut([[X,\Gamma]])$, and let
$\varphi_\eta: X\to X$ realize $\eta$. Then it is clear that $(\varphi_\eta)_\ast\mu$ is a $\Gamma$-invariant Borel probability measure.
Hence $(\varphi_\eta)_\ast\mu=\mu$.
\end{remark}

\begin{definition}
	Let $\aut_X(\Gamma)$ be the subgroup of $\aut([[X,\Gamma]])$ of automorphisms $\theta:\Gamma\to\Gamma$ for which there is a (then unique)
	homeomorphism
	$\phi_\theta:X\to X$ with $\phi_\theta(gx)=\theta(g)\phi(x)$ for every $g\in\Gamma$ and $x\in X$. Further, let $\aut_{(X,\mu)}(\Gamma)$
	be the subgroup $\{\theta\in\aut_X(\Gamma)\mid (\phi_\eta)_\ast\mu=\mu\}$.
\end{definition}

We denote by ${\rm grVect}$ the category of $\mathbb N$-graded vector spaces and by ${\rm grRing}$ the category of $\mathbb N$-graded unital rings. For a probability space $(X,\mu)$ with a $\mu$-preserving $\Gamma$-action we denote the space of $\mu$-square integrable measurable functions by $L^2(X,\mu)$; it becomes a left $\Gamma$-module via $\gamma\cdot f(x)=f(\gamma^{-1}x)$ for $\gamma\in \Gamma$ and $f\in L^2(X,\mu)$. 

The rest of this section is devoted to the proof of the following
two theorems. 

\begin{theorem}\label{thm: full group representation in nilpotent case}
Let $R$ be the orbit equivalence relation of a free and minimal action of a group on a Cantor space $X$. There exists a functor
$H^* \colon m{\mathcal C}(R) \to {\rm grVect},$
that assigns to $(\Gamma,X,\pi,\mu)$ the graded vector space $H^*(\Gamma,L^2(X,\mu))$ of the group cohomology of $\Gamma$ with coefficients in the $\Gamma$-module $L^2(X,\mu)$.

Assume further that $R$ is the orbit equivalence relation of a free and minimal action of a finitely generated nilpotent group on a Cantor space and let $m{\mathcal C}(R)_{\rm nil}$ be the full 
subcategory of $m{\mathcal C}(R)$ whose objects $(\Gamma, X, \pi,\mu)$ have the property that $\Gamma$ is finitely generated and nilpotent\footnote{We do not know whether $m{\mathcal C}(R)_{\rm nil}=m{\mathcal C}(R)$}. 

Then there is a functor
$\Psi \colon m{\mathcal C}(R)_{\rm nil} \to {\rm grRing}$ to graded unital rings that assigns to $(\Gamma,X,\pi,\mu)$ in $m{\mathcal C}(R)_{\rm nil}$ the real cohomology algebra $H^*(\Gamma,\Rb)=\bigoplus_{i\ge 0} H^i(\Gamma,\Rb)$ with the following properties: 
\begin{enumerate}
\item The maps $H^\ast(\Gamma, L^2(X,\mu))\to H^\ast(\Gamma, \Rb)$ induced by integration in the coefficients yield a natural equivalence of functors from $H^\ast$, restricted to $m{\mathcal C}(R)_{\rm nil}$, to $\Psi$. 
\item By restricting $\Psi$ to the automorphism group of an object $(\Gamma,X, \pi,\mu)$ in $m{\mathcal C}(R)_{\rm nil}$ 
   we obtain a homomorphism
\[
	\aut_{\mu}([[X,\Gamma]])\to \aut(H^\ast(\Gamma,\Rb)),
\]
that extends the canonical homomorphism $\aut_{(X,\mu)}(\Gamma)\to \aut(H^\ast(\Gamma;\Rb))$.
\end{enumerate}
Here we denote by $\aut(H^\ast(\Gamma,\Rb))$ the group of graded unital algebra automorphisms of $H^\ast(\Gamma,\Rb)$.
\end{theorem}

The following theorem summarizes the situation in the special case $\Gamma=\Zb^d$. We denote by ${\rm mVect}$ the category whose objects are finite-dimensional vector spaces equipped with a translation invariant Lebesgue measure, and morphisms are given by measure preserving linear maps.

\begin{theorem}\label{thm: full group representation in abelian case}
Let $R$ be the orbit equivalence relation of a free and minimal action of $\Zb^d$ on a Cantor space $X$. There exists a functor $\Psi^1 \colon m{\mathcal C}(R) \to {\rm mVect}$ given on objects by
$\Psi^1(\Zb^d,X,\pi,\mu) := \Rb^d,$ where the volume form is induced by the standard lattice $\Zb^d \subset \Rb^d$. In particular, for every ergodic $\Zb^d$-invariant probability measure on $X$ there is a homomorphism
\[ \aut_{\mu}([[X,\Zb^d]])\to \{ A\in GL_d(\Rb)\mid \det(A)=\pm 1\}
\]
that extends the canonical homomorphism $\aut_{(X,\mu)}(\Zb^d)\to \aut(H^1(\Zb^d;\Rb))$.
\end{theorem}

We also prove the following realization theorem:

\begin{theorem} \label{thm: realization}
Let $d \in \Nb$.
For every matrix $A \in  GL_d(\Rb)$ with $\det(A)=\pm 1$, there exists an equivalence relation $R$ induced by a free $\Zb^d$-action as above and a morphism $\eta=(\varphi_{\eta},\alpha_\eta)$ in the category $m{\mathcal C}(R)$ such that the induced linear map $\Psi^1(\eta)$ is given by left-multiplication with the matrix $A$.
\end{theorem}

\begin{example}
Let $\Zb$ act on the $p$-adic integers $\Zb_p$ by adding~$1$. The product system of $\Gamma=\Zb\times\Zb$ acting on $X= \Zb_p\times \Zb_p$ is a
minimal, uniquely ergodic Cantor system -- unique ergodicity follows from the uniqueness of the Haar measure. In this case, we obtain $\aut(\Zb^2)=\gl_2(\Zb)=\aut_X(\Zb^2)$.
The orbit preserving homeomorphism $\phi_A\colon X\to X$ associated to $A\in\gl_2(\Zb)$ is just matrix multiplication with~$A$.
This shows that that the image of ${\rm Aut}_{\mu}([[X,\Gamma]])$ under the homomorphism of Theorem~\ref{thm: full group representation in abelian case}
contains all matrices with integer entries in this case.
\end{example}

In general, it is still rather mysterious what subgroups of  $\{A\in GL_d(\Rb)\mid \det(A)=\pm 1\}$ can appear as the image of ${\rm Aut}_{\mu}([[X,\Gamma]])$ -- being an intrinsic invariant of the topological orbit equivalence relation together with some ergodic invariant measure. The countable subgroup arising as the image of $H^1(\Zb^d,C(X,\Zb))$ in $\Rb^d$ (given by integration with respect to $\mu$) is always preserved by the action of ${\rm Aut}_{\mu}([[X,\Gamma]])$.

\section{Maps in cohomology induced by morphisms in $m{\mathcal C}(R)$ and ${\mathcal C}(R)$}\label{sec: functors}

Let $\eta =(\varphi_{\eta},\alpha_{\eta})$ be a morphism in the category $m{\mathcal C}(R)$ between $(\Gamma,X,\pi,\mu)$ and $(\Lambda,Y,\sigma,\nu)$. It is clear that $\eta$ is invertible with inverse $\eta^{-1} =(\varphi_{\eta^{-1}},\alpha_{\eta^{-1}})$.
One easily verifies that
\begin{align}\label{eq: inverse cocycle}
\phi_{\eta^{-1}} &=(\phi_\eta)^{-1}\notag\\
g&=\alpha_{\eta^{-1}}(\alpha_\eta(g,x),\phi_\eta(x))
\end{align}
Further, we have
\begin{equation}\label{eq: inverse equivariance cocycle}
	\varphi_\eta(\pi(g^{-1},x))=\sigma(\alpha_\eta(g, g^{-1}x)^{-1}, \varphi_\eta(x))
\end{equation}
since $\varphi_\eta(x)=\varphi_\eta(\pi(gg^{-1},x))=\sigma(\alpha_\eta(g, g^{-1}x),\varphi_\eta(g^{-1}x))$. From now on we will omit to mention the actions $\pi$ and $\sigma$, unless there is any risk of confusion.
From~\eqref{eq: inverse cocycle} and~\eqref{eq: inverse equivariance cocycle} we conclude the identity
\begin{equation}\label{eq: coycle id eta and eta inverse}
	\alpha_{\eta^{-1}}(\alpha_\eta(g,g^{-1}x),\alpha_\eta(g,g^{-1}x)^{-1}\phi_\eta(x))=g.
\end{equation}
We equip $\Gamma^{k+1}\times X$ with the product $\bar \mu$ of the counting measure on $\Gamma^{k+1}$ and $\mu$.
By~\eqref{eq: coycle id eta and eta inverse} the continuous map
	\begin{gather*}
		\Omega^k_\eta\colon \Gamma^{k+1}\times X\to  \Lambda^{k+1}\times Y\\
	\bigl(g_0,\ldots,g_k,x\bigr)\mapsto\bigl(\alpha_\eta(g_0,g_0^{-1}x),\ldots, \alpha_\eta(g_k,g_k^{-1}x), \phi_\eta(x)\bigr)
	\end{gather*}
   is a homeomorphism with inverse given by 	
  \[
  	  \bigl(h_0,\ldots,h_k,y\bigr)\mapsto\bigl(\alpha_{\eta^{-1}}(h_0,h_0^{-1}x),\ldots, \alpha_{\eta^{-1}}(h_k,h_k^{-1}y), \phi_{\eta^{-1}}(y)\bigr)
  \]	
  Since the cocycles are continuous, there is a clopen partition $P$ of $X$ such that for each $A\in P$ the restriction of
  $\Omega^k_\eta$ to $\{(g_0,\ldots, g_k)\}\times A$ composed with the projection to $\Lambda^{k+1}$ is constant. Then $\nu(\Omega^k_\eta(\{(g_0,\ldots, g_k)\}\times A))=\mu(\phi_\eta(A))=\mu(A)=\nu(\{(g_0,\ldots, g_k)\}\times A)$.
  So $\Omega^k_\eta$ is locally measure-preserving and hence measure-preserving, i.e., $\Omega_{\eta*}^k(\bar \mu)= \bar \nu$.

In the sequel, we will consider the reduced group cohomology of a group $\Gamma$ in a Banach $\Gamma$-module. Whereas the ordinary group cohomology in degree $k$ can be defined as 
the quotient of the kernel of the $k$-th differential by the image of the preceding differential in the (homogeneous) bar complex, the \emph{reduced cohomology} is the quotient of the kernel of the $k$-th differential by the closure of the image of the preceding differential with respect to the topology of pointwise convergence on the bar complex. 

Similar maps as in the following lemma were previously studied
in~\cites{shalom, sauer}. 

\begin{lemma}\label{lem: induced map in cohomology}
Let $\eta =(\varphi_{\eta},\alpha_{\eta})$ be a morphism in the category $m{\mathcal C}(R)$ between $(\Gamma,X,\pi,\mu)$ and $(\Lambda,Y,\sigma,\nu)$. Then the maps
\[C^k(\eta)\colon C^k(\Lambda, L^\infty(Y,\nu,\Rb))\to C^k(\Gamma, L^\infty(X,\mu,\Rb)),~C^k(\eta)(f)=f\circ\Omega^k_\eta\]
for a homogenous cochain $f: \Gamma^{k+1}\to  L^\infty(X,\mu,\Rb)$ form a $\Gamma$-equivariant
chain homomorphism and induce contravariant functorial homomorphisms in cohomology
\[
	H^k(\eta)\colon H^k(\Lambda, L^\infty(Y,\nu,\Rb))\to H^k(\Gamma, L^\infty(X,\mu,\Rb)).
\]
Moreover, there are analogous induced contravariant functorial homomorphisms for ordinary and reduced cohomology with coefficients in $L^2(X,\mu,\Zb)$ and $L^2(X,\mu,\Rb)$. Similarly, analogous functors exist on the category ${\mathcal C}(R)$, when we consider cohomology with coefficients in $C(X,\Zb)$ or $C(X,\Rb)$.
\end{lemma}

\begin{proof}[Proof of Lemma~\ref{lem: induced map in cohomology}]
Since $\Omega_\eta^k$ is a homeomorphism, for every tuple $(g_0,\ldots, g_k)\in\Gamma^{k+1}$ there is a finite subset
$F\subset \Lambda^{k+1}$ with $\Omega_\eta^k(\{(g_0,\ldots, g_k)\}\times X)\subset F\times Y$. Together with the fact that
$\Omega_\eta^k$ is measure-preserving it follows that $C^k(\eta)$ is well defined.
It is clear from the definition that $C^k(\eta)$ is compatible with the differential of the homogeneous bar complex.
Next we show that $C^k(\eta)$ restricts to the subcomplexes of $\Gamma$-invariant cochains: The cocycle property of $\alpha_\eta$ implies that
\begin{equation}\label{eq: cocycle property convoluted}
	\alpha_\eta(gg_0, g_0^{-1}x)=\alpha_\eta(g,g_0g_0^{-1}x)\alpha_\eta(g_0,g_0^{-1}x)=\alpha_\eta(g,x)\alpha_\eta(g_0,g_0^{-1}x).
\end{equation}
We have to show that
\[
    C^k(\eta)(f)(gg_0,\ldots,gg_k)(gx)=C^k(\eta)(f)(g_0,\ldots,g_k)(x)
\]
provided $f$ is a $\Gamma$-invariant cochain.
This follows from the following computation:
\begin{align*}
	C^k(\eta)(f)(gg_0,\ldots,gg_k)(gx)
	&= f(\alpha_\eta(gg_0, g_0^{-1}g^{-1}gx),\ldots, \alpha_\eta(gg_k, g_k^{-1}g^{-1}gx))(\phi_\eta(gx))\\
	&= f(\alpha_\eta(gg_0,g_0^{-1}x),\ldots, \alpha_\eta(gg_k,g_k^{-1}x))(\alpha_\eta(g,x)\phi_\eta(x))\\
	&\overset{\eqref{eq: cocycle property convoluted}}{=} f(\alpha_\eta(g,x)\alpha_\eta(g_0,g_0^{-1}x),\ldots, \alpha_\eta(g,x)\alpha_\eta(g_k,g_k^{-1}x))(\alpha_\eta(g,x)\phi_\eta(x))\\
	&= f(\alpha_\eta(g_0,g_0^{-1}x),\ldots, \alpha_\eta(g_k,g_k^{-1}x))(\phi_\eta(x))\\
	&= C^k(\eta)(f)(g_0,\ldots,g_k)(x)\qedhere
\end{align*}
It remains to prove functoriality; this is done in the following lemma.
\end{proof}

\begin{lemma}\label{lem: functoriality}
The assignment $\eta \mapsto H^k(\eta)$ is functorial.
\end{lemma}

\begin{proof}
It is clear that $H^k(\id)=\id$.
Let $\eta = (\varphi_{\eta},\alpha_{\eta})$ and $\theta = (\varphi_{\theta},\alpha_{\theta})$ be morphisms in the category ${\mathcal C}(R)$. Then
\begin{align*}
	\phi_{\eta\circ\theta}&=\phi_{\eta}\circ\phi_{\theta},\text{ and}\\
	\alpha_{\eta\circ\theta}(g,x)&=\alpha_\eta(\alpha_\theta(g,x),\theta(x)).
\end{align*}
The claim follows from the following computation:
\begin{align*}
	C^k(\eta\circ\theta)(f)(g_0, g_1,\ldots)(x)
	&=f(\alpha_{\eta\circ\theta}(g_0,g_0^{-1}x),\ldots)(\eta(\theta(x)))\\
	&=f(\alpha_\eta(\alpha_\theta(g_0,g_0^{-1}x), \theta(g_0^{-1}x)),\ldots)(\eta(\theta(x)))\\
	&\overset{\eqref{eq: inverse equivariance cocycle}}{=}f(\alpha_\eta(\alpha_\theta(g_0,g_0^{-1}x), \alpha_\theta(g_0,g_0^{-1}x)^{-1}\theta(x)),\ldots)(\eta(\theta(x)))\\
	&=C^k(\eta)(f)(\alpha_\theta(g_0,g_0^{-1}x),\ldots)(\theta(x))\\
	&=C^k(\eta)(C^k(\theta)(f))(g_0,\ldots)(x)\qedhere
\end{align*}
\end{proof}

This finishes the proof of existence of various well-behaved cohomological invariants for orbit equivalence relations on Cantor space. Unfortunately, these invariants are typically infinite-dimensional and thus not directly computable. After studying product structures in the next section we will then connect with work of Shalom to see that at least for nilpotent groups $L^2$-cohomology can be identified with the ordinary cohomology of the acting group. This then allows to obtain the desired representation of the category $m{\mathcal C}(R)$ on the category of finite-dimensional vector spaces.
\section{Product structures}

\begin{lemma}\label{lem: multiplicativity cup product}
	Let $\eta =(\varphi_{\eta},\alpha_{\eta})$ be a morphism in the category $m{\mathcal C}(R)$ between $(\Gamma,X,\pi,\mu)$ and $(\Lambda,Y,\sigma,\nu)$. The induced map
	\[H^k(\eta)\colon H^\ast(\Lambda, L^\infty(Y,\nu,\Rb))\to H^\ast(\Gamma, L^\infty(X,\mu,\Rb))
	\]
	is multiplicative with the respect to the cup product.
\end{lemma}

\begin{proof}
This follows right away from the definition of the cup product on the chain level: For cocycles
$f\in C^p(\Lambda, L^\infty(Y,\nu,\Rb))$ and $g\in C^q(\Lambda, L^\infty(Y,\nu,\Rb))$ we have
\[(f\cup g)(h_0,\ldots, h_p, h_{p+1},\ldots, h_{p+q},y)=f(h_0,\ldots,h_p,y)g(h_p,\ldots, h_{p+q},y),\]
for all $h_0,\dots,h_{p+q} \in \Lambda$ and $y \in Y$.
\end{proof}

The cap product in group (co)homology~\cite{brown}*{p.~113} yields
a homomorphism
\[ H^k(\Gamma, L^2(X,\mu,\Rb))\otimes_\Rb H_k(\Gamma, L^2(X,\mu,\Rb))\to H_0(\Gamma, L^2(X,\mu, \Rb)\otimes_\Rb L^2(X,\mu,\Rb)).
\]
We compose this homomorphism with the map induced by the scalar product 
\[L^2(X,\mu,\Rb)\otimes_\Rb L^2(X,\mu,\Rb)\to \Rb\]
to obtain a pairing between cohomology and homology:
\begin{equation}\label{eq: pairing}
\langle\_,\_\rangle\colon H^k(\Gamma, L^2(X,\mu,\Rb))\otimes_\Rb H_k(\Gamma, L^2(X,\mu,\Rb))\to H_0(\Gamma,\Rb)=\Rb.
\end{equation}

This pairing is easily seen to descend to a pairing of reduced cohomology and homology (see also the explicit
formula~\eqref{eq: pairing explicit}):
\[ \langle\_,\_\rangle\colon \bar H^k(\Gamma, L^2(X,\mu,\Rb))\otimes_\Rb H_k(\Gamma, L^2(X,\mu,\Rb))\to H_0(\Gamma,\Rb)=\Rb
\]

\begin{lemma}\label{lem: induced map in homology}
Let $\eta =(\varphi_{\eta},\alpha_{\eta})$ be a morphism in the category $m{\mathcal C}(R)$ between $(\Gamma,X,\pi,\mu)$ and $(\Lambda,Y,\sigma,\nu)$.
For every $k\ge 0$ there is a homomorphism
	\[H_k(\eta)\colon H_k(\Gamma, L^2(X,\mu,\Rb))\to  H_k(\Lambda, L^2(Y,\nu,\Rb))
	\]
	that
	is adjoint to $H^k(\eta)$ in the sense that
	\[
		\langle H^k(\eta)(x),y\rangle=\langle x, H_k(\eta)(y)\rangle
	\]
	for every $x\in H^k(\Lambda, L^2(Y,\nu,\Rb))$ and
	every $y\in H_k(\Gamma, L^2(X, \mu,\Rb))$.
Similar statements holds for homology with coefficients in $C(X,\Zb), C(X,\Rb)$ and $L^2(X,\mu,\Zb)$, where the statement about adjointness only makes sense in presence of invariant measures.
\end{lemma}

\begin{proof}
Let $C_\ast(\Gamma)$ be the (homogeneous) bar complex with
$C_k(\Gamma)=\Zb[\Gamma^{k+1}]$ endowed with the diagonal $\Gamma$-action.
To define $H_k(\eta)$ we rewrite the chain complexes $L^2(X,\mu,\Zb)\otimes_{\Zb\Gamma}C_\ast(\Gamma)$ and
$L^2(X,\mu,\Rb)\otimes_{\Zb\Gamma}C_\ast(\Gamma)$.
Let $L^2(\Gamma^{k+1}\times X,\Rb)^f$ be the subgroup of $L^2(\Gamma^{k+1}\times X,\bar \mu, \Rb)$ that consists of maps
$g\colon \Gamma^{k+1}\times X\to\Rb$  for which there is a finite subset
$F\subset\Gamma^{k+1}$ such that the essential support of $g$ is in $F\times X$.
Let $L^2(\Gamma^{k+1}\times X,\Zb)^f\subset L^2(\Gamma^{k+1}\times X,\Rb)^f$ be the subgroup of essentially $\Zb$-valued
functions. Then
$L^2(\Gamma^{\ast+1}\times X,\Zb)^f$ and $L^2(\Gamma^{\ast+1}\times X,\Rb)^f$  form chain complexes with respect to the differential
\[
	df(\gamma_0,\ldots,\gamma_k,x)=\sum_{i=0}^{k}(-1)^i\sum_{\gamma\in\Gamma}f(\gamma_0,\ldots,\gamma_{i-1},\gamma,\gamma_{i},\ldots, \gamma_k, x).
\]
The homomorphisms
\[
	L^2(X,\mu,\Zb)\otimes_{\Zb}C_\ast(\Gamma)\to L^2(\Gamma^{\ast+1}\times X,\Zb)^f
\]
that map $f\otimes (\gamma_0,\ldots,\gamma_k)$ to
\[
	(\gamma_0',\ldots,\gamma_k',x)\mapsto\begin{cases} f(x) & (\gamma_0',\ldots,\gamma_k')=(\gamma_0,\ldots,\gamma_n)\\
	                                                   0    & \text{otherwise}.
	\end{cases}
\]
are an equivariant chain isomorphism. A similar statement holds for complexes with real-valued functions. Thus we obtain
chain isomorphisms
\begin{gather*}
		L^2(X,\mu,\Zb)\otimes_{\Zb\Gamma}C_\ast(\Gamma)\xrightarrow{\cong} L^2(\Gamma^{\ast+1}\times X,\Zb)^f_\Gamma\\
			L^2(X,\mu,\Rb)\otimes_{\Zb\Gamma}C_\ast(\Gamma)\xrightarrow{\cong} L^2(\Gamma^{\ast+1}\times X,\Rb)^f_\Gamma
\end{gather*}
where the right hand sides are the co-invariants of $L^2(\Gamma^{\ast+1}\times X,\Zb)^f$ and
$L^2(\Gamma^{\ast+1}\times X,\Rb)^f$, respectively.

Let $\eta =(\varphi_{\eta},\alpha_{\eta})$ be a morphism in the category $m{\mathcal C}(R)$ between $(\Gamma,X,\pi,\mu)$ and $(\Lambda,Y,\sigma,\nu)$.
  Now we define
  \begin{gather*}
  		C_k(\eta)\colon L^2(\Gamma^{\ast+1}\times X,\Zb)^f\to L^2(\Lambda^{\ast+1}\times Y,\Zb)^f\\
  		C^k(\eta)(f)=f\circ\Omega^k_{\eta^{-1}},
  \end{gather*}
and similarly for the chain complex with real-valued functions. This is well defined: If $F\times X$ contains the support of $f$ for some
finite subset $F\subset \Gamma^{k+1}$, then $(\Omega^k_{\eta^{-1}})^{-1}(F\times X)=\Omega^k_\eta(F\times X)$ contains
the support of $C_k(\eta)(f)$. But $\Omega^k_\eta(F\times X)$ is compact, thus lies in the product of some finite subset
of $\Lambda^{k+1}$ with $Y$. Since $\Omega^k_{\eta^{-1}}$ is measure-preserving, it also preserves square-integrability.
Analogously to the proof of
Theorem~\ref{lem: induced map in cohomology} one easily verifies that $C_*(\eta)$ is a chain map and descends to the co-invariants.
So we obtain induced maps
\begin{equation*}
	H_\ast(\eta): H_\ast(\Gamma, L^2(X,\mu,\Zb))\to H_\ast(\Lambda, L^2(Y,\nu,\Zb))
\end{equation*}
and similarily for real coefficients.
The adjointness property readily follows from the fact that $\Omega_\eta^k$ is measure-preserving and
the fact that the pairing~\eqref{eq: pairing} is explicitly given on the chain level by
\begin{equation}\label{eq: pairing explicit}
C^k(\Gamma, L^2(X,\mu,\Rb))\otimes_\Rb  L^2(\Gamma^{k+1}\times X,\Rb)^f\to\Rb,~f\otimes g\mapsto \int_{\Gamma^{k+1}\times X}f(z)g(z)d\nu(z).
\qedhere
\end{equation}
\end{proof}

The existence of product structures is of independent interest. They will be used in the next section to finish the proof of Theorem \ref{thm: full group representation in abelian case}.

\section{Conclusion of proofs}

We now finish the proofs of the main results mentioned in Section~\ref{sec: orcat}.
\begin{proof}[Proof of Theorem~\ref{thm: full group representation in nilpotent case}]
We define a functor $\Psi \colon m{\mathcal C}(R)_{\rm nil} \to {\rm grRing}$ on the level of objects as $\Psi(\Gamma,X,\pi,\mu) = H^*(\Gamma,\Rb)$ and on the level of morphisms as follows:

Let $\eta =(\varphi_{\eta},\alpha_{\eta})$ be a morphism in the category $m{\mathcal C}(R)_{\rm nil}$ between $(\Gamma,X,\pi,\mu)$ and $(\Lambda,Y,\sigma,\nu)$.
We define $\Psi(\eta)=(\Psi^k(\eta))_{k\ge 0}$
by declaring $\Psi^k(\eta)$ to be
the unique homomorphism that makes the following diagram
commutative:
\begin{equation}\label{eq: L2-square}
	\xymatrix{
	\bar H^k(\Lambda, L^2(Y,\nu,\Rb))\ar[r]^{\bar H^k(\eta)}_\cong & \bar H^k(\Gamma, L^2(X,\mu,\Rb))\\
	H^k(\Lambda,\Rb)\ar[u]^\cong \ar[r]^{\Psi^k(\eta)}& H^k(\Gamma,\Rb) \ar[u]^\cong
	}
\end{equation}
The vertical maps are induced by inclusion of constant functions in the coefficients.
By~\cite{shalom}*{Theorem~4.1.3} and the ergodicity of $\mu$ we have
$\bar H^k(\Gamma,L^2_0(X,\mu,\Rb))=0$ for the mean value zero functions $L^2_0(X,\mu,\Rb)$.

Since $H^\ast(\Gamma,\Rb)$ is finite dimensional, e.g.~as a consequence of Nomizu's theorem~\cite{oprea}*{Theorem~4.1.3}, reduced cohomology
and (ordinary) cohomology coincide:
\[\bar H^k(\Gamma,\Rb)=H^k(\Gamma,\Rb)\]
Hence the vertical maps are isomorphisms.
Now Lemma~\ref{lem: functoriality} implies that $\Psi$ is functorial: Indeed, 
the functoriality of the lower horizontal arrows in~\eqref{eq: L2-square} follows 
immediately from the functoriality of the upper horizontal arrows since the vertical arrows are isomorphisms. 

Next we derive an explicit formula for $\Psi$. The right vertical map in~\eqref{eq: L2-square} has an obvious left inverse, namely the map induced
by integration in the coefficients, which we denote by
\begin{equation}\label{eq: integration L2 coefficients}
\int_X\colon \bar H^k(\Gamma, L^2(X,\mu,\Rb))\to \bar H^k(\Gamma, \Rb)=H^k(\Gamma, \Rb).
\end{equation}

Since the right vertical map in~\eqref{eq: L2-square} is an isomorphism the map $\int_X$ is its inverse. 
This implies that, for a cocycle $f: \Gamma^{k+1}\to\Rb$, $\Psi^k([f])$
is represented by the cocycle
\begin{equation}\label{eq: explicit formula}
	(g_0,\ldots, g_k)\mapsto \int_X f(\alpha_\eta(g_0, g_0^{-1}x),\ldots, \alpha_\eta(g_k,g_k^{-1}x))d\mu(x).
\end{equation}
We could have defined $\Psi$ via~\eqref{eq: explicit formula} right away but it would have been impossible to prove functoriality without using the fact that the vertical arrows in~\eqref{eq: L2-square} are isomorphisms. The reason we derived~\eqref{eq: explicit formula} after proving functoriality is that we need it now for proving 
that
$\Psi^\ast(\eta)$ is a homomorphism of cohomology algebras. 

To this end, we turn from cohomology with $L^2$-coefficients needed for proving functoriality to cohomology with $L^\infty$-coefficients needed for proving multiplicativity. The cohomology with $L^\infty$-coefficients is an algebra unlike the cohomology with $L^2$-coefficients. 

The map~\eqref{eq: integration L2 coefficients} has an analog for 
$L^\infty$-coefficients:
\begin{equation*}
\int_X\colon \bar H^k(\Gamma, L^\infty(X,\mu,\Rb))\to \bar H^k(\Gamma, \Rb)=H^k(\Gamma, \Rb).
\end{equation*}
From the explicit formula~\eqref{eq: explicit formula} we see that 
the following diagram commutes:
\begin{equation}\label{eq: Linfty-square}
	\xymatrix{
	 H^k(\Lambda, L^\infty(Y,\nu,\Rb))\ar[r]^{ H^k(\eta)} &  H^k(\Gamma, L^\infty(X,\mu,\Rb))\ar[d]^{\int_X}\\
	H^k(\Lambda, \Rb)\ar[r]^{\Psi^k(\eta)}\ar[u] & H^k(\Gamma, \Rb)}
\end{equation}
The left vertical map above is obviously multiplicative with respect to the cohomology algebra structures. By
Lemma~\ref{lem: multiplicativity cup product} $H^k(\eta)$ is
multiplicative. Since $\Gamma$ is finitely generated nilpotent, the map $\int_X$ is multiplicative
by~\cite{sauer}*{Proof of Theorem~5.1.}. If $\eta$ is coming from an
homomorphism from $\Gamma$ to $\Lambda$, then the cocycle $\alpha_\eta$ is constant, and
the explicit formula~\eqref{eq: explicit formula} shows that
$\Psi^k(\eta)$ is the usual map induced by the homomorphism and functoriality of group cohomology.
This finishes the proof.
\end{proof}

\begin{remark}
It is clear on the level of objects that the functor $\Psi$ makes no distinction between $(\Gamma,X,\pi,\mu_1)$ and $(\Gamma,X,\pi,\mu_2)$ if $\mu_1, \mu_2$ are ergodic $\Gamma$-invariant probability measures on $X$. It is unclear if this is also true on the level of morphisms.
\end{remark}


\begin{proof}[Proof of Theorem~~\ref{thm: full group representation in abelian case}]
    By Theorem~\ref{thm: full group representation in nilpotent case}
    there is a homomorphism
    \[\Psi^\ast\colon \aut_\mu([[X,\Zb^d]])\to \aut(H^\ast(\Zb^d, \Rb))\]
    extending the natural homomorphism from $\aut_X(\Zb^d)\subset\gl_d(\Zb)$ to
    $\aut(H^\ast(\Zb^d, \Rb))$. The stated homomorphism of
    Theorem~\ref{thm: full group representation in abelian case} is
    $\Psi^1$. Note that $H^1(\Zb^d,\Rb)\cong \Rb^d$. It remains to
    be shown that $\Psi^1(\eta): H^1(\Zb^d,\Rb)\to H^1(\Zb^d,\Rb)$
    has determinant $\pm 1$ for every $\eta\in\aut_\mu([[X,\Zb^d]])$.

    It is well known that the cohomology algebra of $\Zb^d$ is the exterior algebra over the vector space $H^1(\Zb^d;\Rb)$:
	\begin{equation}\label{eq: exterior algebra}
		H^\ast(\Zb^d;\Rb)\cong \Lambda_\Rb^\ast H^1(\Zb^d;\Rb)
	\end{equation}
	Since $\Psi^\ast(\eta)$ is multiplicative with respect to the
	cup product, the homomorphism
	\[\Rb\cong H^d(\Zb^d,\Rb)\xrightarrow{\Psi^d(\eta)} H^d(\Zb^d,\Rb)\cong\Rb\]
	is given by multiplication with $\det(\Psi^d(\eta))$, which
	is to be shown to lie in $\{1,-1\}$. To this end,
	we show that $\Psi^d(\eta)$ preserves the
	integral lattice $H^d(\Zb^d,\Zb)\subset H^d(\Zb^d,\Rb)$. The pairing
	between cohomology and homology and the fact that $\Zb^d$ is a
	Poincare duality group are the essential ingredients.
	
	Let $a_d\in H^d(\Zb^d, \Zb)\cong\Zb$ be a cohomological
	fundamental class, that is, a generator. Let $x_d\in H_d(\Zb^d,\Zb)$
	be the dual generator. Let $c_d\in \bar H^d(\Zb^d, L^2(X,\mu,\Rb))$
	be the image of $a_d$ under the map induced by inclusion of
	constant functions
	$H^d(\Zb^d, \Rb)\to \bar H^d(\Zb^d,L^2(X,\mu,\Rb))$.
	Recall that the latter map is an isomorphism according to the
	proof of
	Theorem~\ref{thm: full group representation in nilpotent case}.
	In particular, $c_d$ generates the $1$-dimensional vector space $\bar H^d(\Zb^d, L^2(X,\mu,\Rb))$.
	Similarly, we consider the images $y_d\in H_d(\Zb^d,L^2(X,\mu,\Zb))$ and
	$z_d\in H_d(\Zb^d, L^2(X,\mu,\Rb))$ of $x_d$. By Poincare duality
	and ergodicity we obtain that
	\begin{align*}
	H_d(\Zb^d,L^2(X,\mu,\Zb))&\cong H^0(\Zb^d, L^2(X,\mu,\Zb))\cong L^2(X,\mu,\Zb)^{\Zb^d}=\Zb,\\
	H_d(\Zb^d,L^2(X,\mu,\Rb))&\cong H^0(\Zb^d, L^2(X,\mu,\Rb))\cong L^2(X,\mu,\Rb)^{\Zb^d}=\Rb.
	\end{align*}
    So $y_d$ is $\Zb$-module generator, and $z_d$ is a $\Rb$-module
    generator. Since
    $$H_d(\eta)\colon H_d(\Zb^d, L^2(X,\mu,\Rb))\to H_d(\Zb^d, L^2(X,\mu,\Rb))$$ restricts
    to $H_d(\eta)\colon H_d(\Zb^d, L^2(X,\mu,\Zb))\to H_d(\Zb^d, L^2(X,\mu,\Zb))$ by Lemma~\ref{lem: induced map in homology},
    there is $m\in\Zb$ such that $H_d(\eta)(y_d)=m\cdot y_d$,
    thus also $H_d(\eta)(z_d)=m\cdot z_d$.
	Let $r\in\Rb\backslash\{0\}$ be the real number so that
	\[ H^d(\eta)(c_d)=r\cdot c_d. \]
	By the commutativity of the diagram~\eqref{eq: L2-square}, 
	for $\Psi^d(\eta)$ to preserve the integral lattice it suffices to show that $r\in \{1,-1\}$. Clearly, we have
	$\langle c_d, z_d\rangle=1$, and thus by
	Lemma~\ref{lem: induced map in homology}
	\[ r=\langle H^d(\eta)(c_d), z_d\rangle=\langle c_d, H_d(\eta)(z_d)\rangle=\langle c_d, m\cdot z_d\rangle=m\in\Zb.\]
	
    By an analogous argument for the inverse $\Psi^{-1}$
    one also get $1/r\in\Zb$ from which $r\in\{1,-1\}$ follows.
\end{proof}

\begin{proof}[Proof of Theorem \ref{thm: realization}]
	First of all, we argue that each matrix $A \in GL_d(\Rb)$ with $\det(A) = \pm 1$ gives rise to a bi-Lipschitz equivalence of $\Zb^d$, which is unique up to bounded distance. We make use of the group structure of $GL_d(\Rb)$ at this point. Every matrix in $\{g \in GL_d(\Rb) \mid \det(g) = \pm 1\}$ is a product of elementary matrices and the matrix ${\rm diag}(-1,1,\dots,1)$. For each of the these matrices, it is obvious that a bi-Lipschitz equivalence of $\Zb^d$ exists, which is at bounded distance. Indeed, each elementary matrix is a shearing transformation essentially of the form $(x,y) \to (x,y+\lambda x)$ for some $\lambda \in \Rb$. This map can be approximated by $(x,y) \to (x,y+\lfloor \lambda x \rfloor)$, which is bijective on $\Zb^2$ and bounded distance to the original map. For $A\in GL_d(\Rb)$ let $f_A\colon \Zb^d\to \Zb^d$ be a choice of a bi-Lipschitz
equivalence which is bounded distance from $A$.
  We now follow the construction in the proof of Theorem~\ref{main1} which produces a topological orbit equivalence between two actions of $\Zb^d$ on a Cantor set~$X$ out of $f_A$; both actions are minimal Cantor systems. We denote the orbit equivalence relation on ~$X$ by~$R_A$. The orbit equivalence yields a morphism $\eta$ in the category $m{\mathcal C}(R_A)$. One easily verifies by inspecting the construction in Theorem~\ref{main1} that for the associated
  cocycle $\alpha_\eta$ there is a constant $C>0$ such that
  \begin{equation}\label{eq: uniformly close}
      |\alpha_\eta(g, x)-A(g)|<C~~~\text{ for all $g\in \Zb^d$ and $x\in X$}.
  \end{equation}
The map $\Psi^1(\eta)$ is, on the level of homogeneous cochains, given by the formula~\eqref{eq: explicit formula}.
In the inhomogeneous picture, a $1$-cochain is a map $f\colon \Zb^d\to\Rb$ (corresponding to the invariant homogeneous $1$-cochain $(g_0,g_1)\mapsto
f(g_1-g_0)$). Moreover, $1$-cocycles are exactly the homomorphisms $\Zb^d\to \Rb$. The map $\Psi^1(\eta)$ for inhomogeneous $1$-cochains is given by
\[\bigl(f\colon\Zb^d\to\Rb\bigr) \mapsto \bigl( g\to \int_X f(\alpha_\eta(g, g^{-1}x))d\mu(x)\bigr).\]
Since cocycles are mapped to cocycles, the right hand side is a homomorphism $\Zb^d\to\Rb$ provided $f$ is a homomorphism.
Further, the map
\[\bigl(f\colon\Zb^d\to\Rb\bigr) \to f(A(g))\]
induces multiplication with $A^t$ on the first cohomology. Now let $f:\Zb^d\to\Rb$ be a homomorphism. Since $f$ is given by taking the standard scalar product with a real vector, \eqref{eq: uniformly close} implies that there is a constant $D>0$ such that for every $g\in \Zb^d$ we have
\[ |\int_X f(\alpha_\eta(g, g^{-1}x))d\mu(x)-f(A(g))|<D. \]
Since homomorphisms $\Zb^d\to\Rb$ that are a bounded distance apart have to be identical, the claim follows.
\end{proof}

\section{Further directions} \label{sec: f	d}

It would be worthwile to study category-theoretical concepts in ergodic 
theory, in particular in orbit equivalence theory, further. The measured orbit category is only a starting point. 
Indeed, it is fairly straightforward to see that the actual structure that should be considered instead of the measured orbit equivalence category is a strict 2-category. Here, the underlying (classical) category is just the measured orbit equivalence category and the 2-morphisms between ergodic-theoretical cocycles are given by the usual notion of equivalence in the $1$-cohomology for such cocycles. This behaves well with the structures that we explored so far. In particular, if two ergodic-theoretical cocycles define the same $1$-cohomology class, then the induced map on group cohomology is the same. We also obtain that various homomorphisms from automorphism groups are actually well-defined on a suitable outer automorphism group. This will all be subject of further study.

Another direction is the question of possible generalizations of the important result of Boyle-Tomiyama \cite{MR1613140} that topologically orbit equivalent Cantor minimal $\mathbb Z$-actions are flip conjugate. Consider a Cantor minimal $\mathbb Z^d$-action. At least in the uniquely ergodic case, any orbit equivalence $\alpha \colon \mathbb Z^d \times X \to \mathbb Z^d$ gives rise due to our construction to a uniquely determined matrix in $\{ A \in GL_d(\mathbb R), \det(A) = \pm 1\}$. Let us denote this invariant by $A(\alpha)$ for now. It is natural to ask if the cocycle, considered as a map $X \to {\rm Bij}(\mathbb Z^d,\mathbb Z^d)$ is uniformly close in the euclidean distance to the map which is given by multiplication with the matrix $A(\alpha)$. Here, ${\rm Bij}(\mathbb Z^d,\mathbb Z^d)$ denotes the space of zero-preserving (bi-Lipschitz) bijections.

In any case, minimal subsystems of ${\rm Bij}(\mathbb Z^d,\mathbb Z^d)$ with respect to the natural $\mathbb Z^d$-action seem to be a fundamental object of study.

\section*{Acknowledgments}
This research was started and the first sections were completed during a visit at the AIM in Palo Alto in September 2011. We thank this institution for its warm hospitality. K.M.\ was supported  by the NSA grant  H98230CCC5334. A.T.\ was supported by ERC Starting Grant No.\ 277728.

We thank the referee for a very careful and helpful report.


\begin{thebibliography}{ZZZ}

\bib{MR1613140}{article}{
   author={Boyle, Mike},
   author={Tomiyama, Jun},
   title={Bounded topological orbit equivalence and $C^*$-algebras},
   journal={J. Math. Soc. Japan},
   volume={50},
   date={1998},
   number={2},
   pages={317--329},
}


\bib{brown}{book}{
   author={Brown, Kenneth S.},
   title={Cohomology of groups},
   series={Graduate Texts in Mathematics},
   volume={87},
   publisher={Springer-Verlag, New York-Berlin},
   date={1982},
   pages={x+306},
}

\bib{MR2730576}{article}{
   author={Dymarz, Tullia},
   title={Bilipschitz equivalence is not equivalent to quasi-isometric
   equivalence for finitely generated groups},
   journal={Duke Math. J.},
   volume={154},
   date={2010},
   number={3},
   pages={509--526},
}
	
\bib{furmanOE}{article}{
   author={Furman, Alex},
   title={Orbit equivalence rigidity},
   journal={Ann. of Math. (2)},
   volume={150},
   date={1999},
   number={3},
   pages={1083--1108},
}

\bib{mercurial}{article}{
   author={Gaboriau, Damien},
   title={Mercuriale de groupes et de relations},
   language={French, with English and French summaries},
   journal={C. R. Acad. Sci. Paris S\'er. I Math.},
   volume={326},
   date={1998},
   number={2},
   pages={219--222},
}
		
\bib{MR1919400}{article}{
   author={Gaboriau, Damien},
   title={On orbit equivalence of measure preserving actions},
   conference={
      title={Rigidity in dynamics and geometry},
      address={Cambridge},
      date={2000},
   },
   book={
      publisher={Springer, Berlin},
   },
   date={2002},
   pages={167--186},
}

\bib{gaboriau-betti}{article}{
   author={Gaboriau, Damien},
   title={Invariants $l^2$ de relations d'\'equivalence et de groupes},
   language={French},
   journal={Publ. Math. Inst. Hautes \'Etudes Sci.},
   number={95},
   date={2002},
   pages={93--150},
}

\bib{MR2393431}{article}{
   author={Giordano, Thierry},
   author={Matui, Hiroki},
   author={Putnam, Ian F.},
   author={Skau, Christian F.},
   title={Orbit equivalence for Cantor minimal $\Bbb Z^2$-systems},
   journal={J. Amer. Math. Soc.},
   volume={21},
   date={2008},
   number={3},
   pages={863--892},
}

\bib{MR2563761}{article}{
   author={Giordano, Thierry},
   author={Matui, Hiroki},
   author={Putnam, Ian F.},
   author={Skau, Christian F.},
   title={Orbit equivalence for Cantor minimal $\Bbb Z^d$-systems},
   journal={Invent. Math.},
   volume={179},
   date={2010},
   number={1},
   pages={119--158},
}

\bib{gromov}{article}{
   author={Gromov, Misha},
   title={Asymptotic invariants of infinite groups},
   conference={
      title={Geometric group theory, Vol.\ 2},
      address={Sussex},
      date={1991},
   },
   book={
      series={London Math. Soc. Lecture Note Ser.},
      volume={182},
      publisher={Cambridge Univ. Press, Cambridge},
   },
   date={1993},
   pages={1--295},
}


\bib{Hj2006}{article}{
   author={Hjorth, Greg},
   author={Molberg,Matt},
   title={Free continuous actions on zero-dimensional spaces},
   journal={Topology Appl.},
   volume={153},
   date={2006},
   number={7},
   pages={1116--1131},
}

\bib{MR3329727}{article}{
   author={Kapovich, Michael},
   title={Lectures on quasi-isometric rigidity},
   conference={
      title={Geometric group theory},
   },
   book={
      series={IAS/Park City Math. Ser.},
      volume={21},
      publisher={Amer. Math. Soc., Providence, RI},
   },
   date={2014},
   pages={127--172},
}
	

\bib{Me2011}{article}{
   author={Medynets, Konstantin},
   title={Reconstruction of orbits of Cantor systems from full groups},
   journal={Bull. Lond. Math. Soc.},
   volume={43},
   date={2011},
   number={6},
   pages={1104--1110},
}

\bib{MR1326733}{article}{
   author={Papasoglu, Panagiotis},
   title={Homogeneous trees are bi-Lipschitz equivalent},
   journal={Geom. Dedicata},
   volume={54},
   date={1995},
   number={3},
   pages={301--306},
}

%
\bib{sauer}{article}{
   author={Sauer, Roman},
   title={Homological invariants and quasi-isometry},
   journal={Geom. Funct. Anal.},
   volume={16},
   date={2006},
   number={2},
   pages={476--515},
}
\bib{sauer+schroedl}{article}{
   author={Sauer, Roman},
   author={Schr\"odl, Michael},
   title={Vanishing of $\ell^2$-Betti numbers of locally compact groups as an invariant of coarse equivalence},
   note={Preprint, arXiv:1702.01685}
   date={2017},
}


\bib{shalom}{article}{
   author={Shalom, Yehuda},
   title={Harmonic analysis, cohomology, and the large-scale geometry of
   amenable groups},
   journal={Acta Math.},
   volume={192},
   date={2004},
   number={2},
   pages={119--185},
}
\bib{oprea}{book}{
   author={Tralle, Aleksy},
   author={Oprea, John},
   title={Symplectic manifolds with no K\"ahler structure},
   series={Lecture Notes in Mathematics},
   volume={1661},
   publisher={Springer-Verlag, Berlin},
   date={1997},
   pages={viii+207},
}
	
\bib{Whyte}{article}{
   author={Whyte, Kevin},
   title={Amenability, bi-Lipschitz equivalence, and the von Neumann
   conjecture},
   journal={Duke Math. J.},
   volume={99},
   date={1999},
   number={1},
   pages={93--112},
}
\end{thebibliography}
\end{document}